\documentclass[12pt]{article}   

\usepackage[latin1]{inputenc}
\usepackage[T1]{fontenc}
\usepackage{amsmath}
\usepackage{amsfonts}
\usepackage{amssymb}
\usepackage{epic}
\usepackage{eepic}
\usepackage{graphicx}
\usepackage{tikz}

\newtheorem{theorem}{Theorem} 
\newtheorem{corollary}[theorem]{Corollary}

\newtheorem{example}[theorem]{Example}

\newtheorem{lemma}[theorem]{Lemma}

\newtheorem{remark}[theorem]{Remark}
\newtheorem{definition}[theorem]{Definition}

\newenvironment {proof} {{\it Proof.}}{\hspace*{\fill}$\Box$\par\vspace{4mm}}

\newcommand{\sign}{\mbox{\textrm  sign$\,$}}
\newcommand{\mc}{\mathcal}
\newcommand{\mb}{\mathbb}

\newcommand{\ab}{{\boldsymbol{a}}}
\newcommand{\Ab}{{\boldsymbol{A}}}
\newcommand{\Bb}{{\boldsymbol{B}}}
\newcommand{\cb}{{\boldsymbol{c}}}
\newcommand{\db}{{\boldsymbol{d}}}
\newcommand\bdelta{{\boldsymbol \delta}}
\newcommand\bnull{{\boldsymbol 0}}
\newcommand\bone{{\boldsymbol 1}}

\newcommand{\CC}{{\mathbb C}}
 \newcommand{\NN}{{\mathbb N}}
 \newcommand{\RR}{{\mathbb R}}
  \newcommand{\ZZ}{{\mathbb Z}}

\title{Subdivision schemes, network flows and linear optimization}

\author{
 Maria Charina\\
 Institute for Applied Mathematics \\
 Dortmund University of Technology\\
 D-44221 Dortmund, Germany
 \and
 Geir Dahl\\
 Departments of Mathematics, CMA\\
 University of  Oslo\\
 P.O. Box 1053 Blindern, 0316 OSLO, Norway\\
 }

\begin{document}

\maketitle

\begin{abstract}
 We link regularity and smoothness analysis of
 multivariate vector subdivision schemes with network flow
 theory and with special linear optimization problems. This
 connection allows us to prove the existence of
 what we call optimal difference masks that posses crucial properties
 unifying the regularity analysis of univariate and multivariate subdivision schemes.
 We also provide efficient optimization algorithms for
 construction of such optimal masks.
 Integrality of the corresponding optimal values leads to
 purely analytic proofs of $C^k-$regularity of subdivision.
\end{abstract}

\section{Introduction}

 There is a large variety of results in the literature that study
 H\"older and Sobolev regularity and other important properties of scalar
 and vector, univariate and multivariate subdivision schemes, see
 \cite{CDM91, DLsurvey, Reif} and references therein.

 Such schemes are recursive algorithms for mesh refinement and, in the regular case, are
 based on the repeated application of the so-called subdivision operator
$$
 S_\Ab:\ell^{n}(\ZZ^s) \rightarrow \ell^{n}(\ZZ^s),
 \quad S_\Ab \cb(\alpha)=
\sum_{\beta \in \ZZ^s} A(\alpha-M\beta)c(\beta),
 \quad \alpha \in \ZZ^s.
$$

 The efficiency of such algorithms is guaranteed by their
 locality, indeed the so-called subdivision mask $\Ab=\{A(\alpha), \
 \alpha \in \ZZ^s\}$
 is usually finitely supported. The topology of the mesh is
 encoded in the dilation matrix $M \in \ZZ^{s \times s}$. For
 details on various applications of subdivision schemes see e.g.
 \cite{CabrelliHM, Christensen,  ChuiV}.

 The methods for regularity analysis of such schemes are
 either based on the so-called joint spectral radius approach \cite{CJR02, DL1992, J95}
 or on the restricted radius approach \cite{CDM91,CCS04}. The results of \cite{C11}
 unify
these approaches and show that both characterize the
 regularity of subdivision in terms of the same quantity: either
 called the joint spectral radius (JSR) of a certain family of square
 matrices derived from the subdivision mask $\Ab$, or the restricted
 spectral radius (RSR) of an associated linear operator of the
 difference subdivision scheme also derived from $\Ab$. H\"older
 or Sobolev regularity of subdivision is characterized in terms of
 $\infty-$JSR or $2-$JSR, respectively. The question of exact
 computation on $2-$JSR in the subdivision context has been
 extensively studied in \cite{H003, JJ03}. The
 numerical methods for estimation of the $\infty-$JSR differ and its computation,
 in general, is an NP-hard problem \cite{Blondel}. Recent theoretical results and
 numerical tests in
\cite{ProtasovGug}
  lead to exact computation of $\infty-$JSR for a wide class of
 families of matrices and rely on the special choice of the
 so-called extremal matrix norm. There are also various results on
 numerical methods for computation of $\infty-$JSR for particular subdivision
 schemes e.g. \cite{Goodman, Reif_4point}.

 We are interested in pursuing further the idea in
 \cite{C11, CCS04} of using optimization methods for estimation of
 $\infty-$JSR. In sections \ref{sec:dual} and \ref{section:optimal_mask_2_multi},
we show that these optimization problems are of a very
special type, namely, they are network flow problems, or, in
general, equivalent to special linear optimization problems. The
properties of network flow problems allow for exact computation of
what we call the optimal first difference subdivision mask, see
section \ref{sec:optimal_1}. The
spectral properties of the corresponding optimal scheme characterize
the convergence of $S_\Ab$. For
such optimal masks the sufficient condition derived in \cite{D92}
for the convergence of $S_\Ab$ also becomes necessary and, thus,
coincides with the characterization of convergence given in
\cite{C11, CCS04}. The advantage of working with optimal difference
masks is twofold. Firstly, the computation of the norm of the
associated subdivision operator is straightforward, see
\cite{D92}, and exact, if all entries $A(\alpha)$ of the subdivision mask are rational. Thus, it allows for analytic arguments in
the convergence proofs. Secondly, the proof of the existence of
such optimal masks bridges the gap between the convergence
analysis of univariate and multivariate schemes $S_\Ab$, i.e., the
non-restricted and restricted norms of the corresponding difference operator
coincide even in the multivariate case. In section 5, we also
prove the existence of optimal masks for higher order difference schemes
and
provide an algorithm for their construction. In
section \ref{sec:examples}, we illustrate our results with several examples.

\section{Background and notation}

In this paper we make use of the following notation.

\begin{description}
\item[-] By $\epsilon_\ell$, $\ell=1, \dots,s$, we denote the standard unit vectors of $\RR^s$.
\item[-] For $\alpha \in \NN_0^s$ we define $|\alpha|=\alpha_1+\ldots + \alpha_s$.
\item[-]  In the multi-index notation
we have $z^\alpha=z_1^{\alpha_1} \cdots z_s^{\alpha_s}$, $z \in \CC^s$ and $\alpha \in \NN_0^s$.
\item[-] The eigenvalues of the dilation matrix $M \in \ZZ^{s \times s}$ are all greater than $1$ in absolute value.
\item[-] Vector sequences $\cb \in \ell^{n}(\ZZ^s)$ indexed by $\ZZ^s$, i.e. functions from $\ZZ^s$ into $\RR^{n}$,
are denoted by boldface letters. Matrix sequences $\Ab \in \ell^{n \times m}(\ZZ^s)$ indexed by $\ZZ^s$, i.e. functions from $\ZZ^s$ into $\RR^{n \times m}$, are denoted by boldface capital letters.
The space of such sequences with finitely many non-zero elements is denoted by $\ell^{n}_0(\ZZ^s)$ or $\ell^{n \times m}_0(\ZZ^s)$, respectively.
\item[-] The finitely supported matrix sequence $\bdelta_{I_n} \in \ell_0^{n \times n}(\ZZ^s)$  is defined by
$$
 \delta_{ I_n}(\alpha):=\left\{
 \begin{array}{cl} I_n, & \alpha = 0, \\
 0, & \alpha \in \ZZ^s \setminus \{ 0 \}.
 \end{array} \right.
$$
\item[-] The norm on the Banach space $\ell^n_\infty(\ZZ^s)$  is given by
$$
 \|\cb\|_\infty=\sup_{\alpha \in \ZZ^s} \|c(\alpha)\|_\infty.
$$
\item[-] For $C=(C_{i,j}) \in \RR^{n \times n}$  we define  $|C|=(|C_{i,j}|) \in \RR^{n \times n}$.
\item[-] For a real number $a$ we define $a^+=\max\{a,0\}$ and
$a^-=-\min\{a,0\}$.
\end{description}

\subsection{Subdivision schemes}
In this subsection we recall some basic facts about multivariate subdivision schemes.

Let $\Ab \in \ell^{n \times n}_0(\ZZ^s)$, the so-called subdivision mask, be given and be supported on $\{0, \ldots, N\}^s$, $N \in \NN$.
A subdivision scheme
\begin{equation} \label{def:subdivision_recursion}
   \cb^{[k+1]}(\alpha)=S_\Ab \cb^{[k]}(\alpha), \quad k \in \NN_0, \quad \cb^{[0]} \in \ell^n(\ZZ^s), \quad \alpha \in \ZZ^s,
\end{equation}
is a repeated application of the so-called subdivision operator
\begin{equation}\label{def:subdivision_operator}
 S_\Ab: \ell^n(\ZZ^s) \rightarrow \ell^n(\ZZ^s), \quad (S_\Ab \cb)(\alpha)=\sum_{\beta \in \ZZ^s} A(\alpha-M\beta) c(\beta).
\end{equation}
Equivalently, the recursion in \eqref{def:subdivision_recursion} can be written as
$$
  \cb^{[k+1]}=S^{k+1}_\Ab \cb^{[0]}, \quad k \in \NN_0, \quad \cb^{[0]} \in \ell^n(\ZZ^s),
$$
where the iterated operator $S^{k+1}_\Ab$ is defined similarly as in \eqref{def:subdivision_operator} by
replacing $\Ab$ by the so-called iterated mask $\Ab^{[k+1]}$ given by
\begin{equation} \label{def:iterated_mask}
 A^{[k+1]}(\alpha)=\sum_{\beta \in \ZZ^s} A^{[k]}(\beta)A(\alpha-M\beta),
 \quad k \in \NN_0, \quad \alpha \in \ZZ^s, \quad \Ab^{[0]}= \bdelta_{ I_n}.
\end{equation}
The definition of $S_\Ab$ in \eqref{def:subdivision_operator}
implies that we have $|\hbox{det}(M)|$ different subdivision
rules, due to $\alpha=\varepsilon+M\beta$, $\varepsilon \in \Xi
\simeq \left(\ZZ^s / M \ZZ^s \right)$ and $\beta \in \ZZ^s$. The
set $\Xi$ is usually called the set of representatives of the
equivalence classes $\ZZ^s / M \ZZ^s$.  In the simplest case,
$s=1$, $n=1$ and $M=2$, we have $\Xi \simeq \{0,1\}$. Thus, we
have different subdivision rules for odd and even $\alpha \in
\ZZ$. With a slight abuse of notation we denote the subdivision
scheme also by $S_\Ab$.

We say that the subdivision scheme $S_\Ab$ is convergent, if for
any starting sequence $\cb \in \ell_\infty^n(\ZZ^s)$, there exists
a uniformly continuous vector-valued function $f_\cb \in
(C^0(\RR^s))^n$ such that
\begin{equation*}
\label{def:convergenceFunction}
 \lim_{r \rightarrow \infty} \sup_{\alpha \in \ZZ^s} \| f_\cb(M^{-r}\alpha) -
  S_\Ab^r \cb(\alpha)\|_\infty =0.
\end{equation*}

To distinguish between scalar and vector subdivision schemes we denote their masks by $\ab$ and
$\Ab$, respectively.

\subsection{Linear optimization and network flows} \label{subsec:intro_LP}
To familiarize the reader with terminology used in this paper, in this subsection, we introduce some basic notions
from the theory of network flows and linear optimization.

A linear optimization (linear programming, LP) problem  (\cite{Vanderbei08}) is to maximize, or minimize, a linear function of $n$ variables subject to a finite number of linear constraints, each being a linear equation or a linear inequality. Since each equation can be written as two inequalities, every LP problem may be written as
\begin{equation} \label{general_LP}
    \mbox{\textup maximize $d^Tx\;$ subject to $C x \le b,$}
\end{equation}
where $C$ is an $m \times n$ matrix, $b$ and $d$ are (column) vectors  of suitable dimensions, and $x \in \mb{R}^n$ is the vector containing the $n$ optimization variables $x_1 x_2, \ldots, x_n$. A {\em feasible} solution of \eqref{general_LP} is a vector $x \in \mb{R}^n$ satisfying all the constraints. A feasible solution is called {\em optimal}, if no other feasible solution attains a larger value on the objective function $d^Tx$. There are practical, efficient algorithms for solving LP problems, see \cite{Vanderbei08}. For LP problems there exist a powerful duality theory, which associates  to every  LP problem  another LP problem, called the {\em dual} problem, and shows close connections between these two problems. In particular, the two problems have the same optimal value (under a weak assumption). If we take the dual twice, we are back to the original problem. We discuss duality in more detail for the special case of network flow problems.

A special class of LP problems, arising in several applications and useful in the context of subdivision, consists of  {\em network flow} problems. Given a directed graph $G=(V,E)$ with vertex set $V$ and edge set $E$ (where each edge is an ordered pair of vertices), real numbers $b_v$ for each $v \in V$ and real numbers (costs) $c_{uv}$ for each $(u,v) \in E$; we write $uv$ sometimes instead of $(u,v)$. The {\em minimum cost network flow} problem is the following LP problem
\begin{equation}
 \label{eq:MCNF}
 \begin{array}{lcl}
   \mbox{\rm minimize} &     \displaystyle \sum_{(u,v) \in E} c_{uv}f_{uv} \\
   \mbox{\rm subject to} \\
        &   \displaystyle \sum_{u: (v,u) \in E} f_{vu}  -
            \sum_{u: (u,v) \in E} f_{uv} = b_v & (v \in V)
            \\ \\*[\smallskipamount]
        & f_{uv} \ge 0  \quad &((u,v) \in E).
 \end{array}
\end{equation}
The interpretation of \eqref{eq:MCNF} is as follows: The variable $f_{uv}$  represents the flow from $u$ to $v$ along
the edge $(u,v)$, and $c_{uv}$ is the unit cost of sending this  flow, so the objective function represents the total cost.
The constraints represent flow balance at every vertex: the total flow leaving a vertex $v$ minus the total flow into the same vertex equals the given number $b_v$, which is called the supply at $v$. Finally, flows $f_{uv}$ are required to be nonnegative at each edge. Often one also has upper bounds (capacities) on flows, but we do not need this here. We assume that problem \eqref{eq:MCNF} has a feasible solution (conditions that guarantee this are known, and require that $\displaystyle \sum_{v \in V} b_v=0$,  see \cite{AMO93}). This condition on $b_v$ has a natural interpretation in the subdivision context,
see Remark \ref{rem:2}. The {\em dual} of the minimum cost network flow problem (\ref{eq:MCNF}) is
\begin{equation}
 \label{eq:dualMCNF}
 \begin{array}{lcl}
   \mbox{\rm maximize} &    \displaystyle \sum_{v \in V} b_v \,x_v \\
   \mbox{\rm subject to} \\
        & x_u-x_v \le c_{uv}&((u,v) \in E).
 \end{array}
\end{equation}
In this problem the variables $x_v$ are associated with vertices and may be interpreted as a kind of potential. Now, assume $f_{uv}$ ($(u,v) \in E$) and $x_v$ ($v \in V$) are feasible solutions of (\ref{eq:MCNF}) and (\ref{eq:dualMCNF}), respectively. Then
\[
    \sum_{(u,v)} c_{uv} f_{uv} \ge \sum_{(u,v)} (x_u-x_v)f_{uv} = \sum_v (\sum_u f_{vu}-\sum_u f_{uv})x_u = \sum_v b_vx_v.
\]
This shows {\em weak duality} which says that the optimal value of (\ref{eq:MCNF}) is not smaller than the optimal value of (\ref{eq:dualMCNF}). Actually, a stronger result, the {\em  duality theorem},  says that these two optimal values are equal. This fact is exploited in very efficient algorithms for solving the minimum cost network flow problem, see \cite{AMO93}. An important result is the {\em integrality theorem} which says that, if each $b_v$ is an integer,  then problem (\ref{eq:MCNF}) has an optimal solution where each $f_{uv}$ is an integer. We then say that the optimal solution is {\em integral}. This integrality property is due to special properties of the coefficient matrix of the flow balance equations, and for other classes of LP problems it may happen that no optimal solution is integral (i.e., in every optimal solution,  at least one variable is not an integer).

\section{First difference subdivision schemes and network flow problems}

In this section we show that optimization problems considered in
\cite{CCS04} for convergence analysis of subdivision are of a very
special type, namely, they are network flow problems.

\subsection{Scalar case} \label{sec:1_diff_scalar}
For simplicity of presentation we start with the scalar
multivariate case, i.e. $n=1$. The subdivision mask $\ab \in
\ell_0(\ZZ^s)$ is a finitely supported sequence of real numbers
$a(\alpha)$, $\alpha \in \ZZ^s$. The convergence analysis of such
subdivision schemes relies on what we call the backward difference
operator $
 \nabla: \ell(\ZZ^s) \rightarrow \ell^s(\ZZ^s)$ given by
\begin{equation}\label{def:nabla}
  \nabla=\left( \begin{array}{c} \nabla_1\\ \vdots \\ \nabla_s \end{array} \right), \quad
 \left(\nabla_\ell \cb \right) (\alpha)=c(\alpha)-c(\alpha-e_\ell), \quad \alpha \in \ZZ^s, \quad \cb \in \ell(\ZZ^s).
\end{equation}
A matrix sequence $\Bb \in \ell^{s \times s}(\ZZ^s)$ that satisfies
\begin{equation} \label{def:S_AvsS_B}
 \nabla S^r_\ab= S^r_\Bb \nabla, \quad r \in \NN,
\end{equation}
defines the so-called difference subdivision operator $S_\Bb : \ell^{s \times s}(\ZZ^s) \rightarrow \ell^{s \times s}(\ZZ^s)$ by
\begin{equation}\label{def:SB}
 S_\Bb \db (\alpha)=\sum_{\beta \in \ZZ^s} B(\alpha-M\beta) d(\beta), \quad \db \in \ell^s(\ZZ^s).
\end{equation}
The existence of $\Bb$ is equivalent to the fact that $\ab$ satisfies sum rules of order $1$, see \cite{S02} for
details. The assumption that $\ab$ satisfies sum rules of order $1$ is by no means restrictive, as it is
also a necessary condition for convergence of $S_\ab$, see e.g. \cite{CDM91, J95}.

We index the entries $B_{j,\ell}(\alpha)$ of the matrices
$B(\alpha)$ by $j=1, \dots, s$ and by $\ell=1, \ldots, s$ to match
the indexing of the entries $\nabla_\ell$ of the difference
operator $\nabla$. One of the approaches for characterizing
convergence of subdivision schemes studies the spectral properties
of the operator $S_\Bb$. Let $r \in \NN$. The results of
\cite{D92} use the non-restricted norm
$$
 \|S^r_\Bb \|_\infty=\max_{\varepsilon \in \Xi_r} \left\|\sum_{\beta \in \ZZ^s} \left|B^{[r]}(\varepsilon-M^r\beta) \right| \right\|_\infty, \quad
 \Xi_r \simeq \left(\ZZ^s / M^r \ZZ^s \right),
$$
to derive sufficient conditions for convergence of the subdivision
scheme $S_\ab$.  In \cite{CCS04} the authors use the restricted
norm
\begin{equation}\label{def:nonresticted_norm}
 \|S^r_\Bb|_\nabla\|_\infty=\max_{\|\nabla \cb\|_\infty=1} \|S^r_\Bb \nabla \cb\|_\infty
\end{equation}
to characterize the convergence of $S_\ab$. Due to \eqref{def:S_AvsS_B}, the operator $S_\Bb$ maps the difference subspace
$\nabla \ell(\ZZ^s)$ into itself and, thus, its restriction $S^r_\Bb|_\nabla$ to  $\nabla \ell(\ZZ^s)$ is well-defined.

Define $K=\{-N-1,\ldots,0\}^s$. Due to $\Bb \in \ell^{s \times s}_0(\ZZ^s)$ and by the periodicity of the operator $S_\Bb$, we have
\begin{eqnarray} \label{def:lin_prog_1}
   \|S^r_\Bb|_\nabla\|_\infty =\max_{\|\nabla \cb|_{K}\|_\infty=1}\max_{\varepsilon \in \Xi_r}  \left\| \sum_{\beta \in K}
   B^{[r]}(\varepsilon-M^r\beta) (\nabla \cb)(\beta) \right\|_\infty.
\end{eqnarray}
The problem of computing of $\|S^r_\Bb|_\nabla\|_\infty$ in
\eqref{def:lin_prog_1} consists of several linear optimization problems for the finitely many unknowns
$c(\beta)$, $\beta \in \{-N-2, \ldots, 0\}^s$. Indeed, to compute
the maximum in \eqref{def:lin_prog_1}, it suffices,  for each pair
$(\varepsilon,j) \in \Xi_r \times\{1,\ldots,s\}$, to solve the
linear optimization problem
\begin{equation} \label{def:lin_prog_2}
 \begin{array}{lcl}
   \max &  \displaystyle \sum_{\beta \in K} \sum_{\ell=1}^s B_{j,\ell}^{[r]}
   (\varepsilon-M^r\beta) (\nabla_\ell \cb)(\beta) \\
   \mbox{\rm subject to} \\
        & -1 \le c(\beta)-c(\beta-\epsilon_\ell) \le 1, \quad \beta \in K, \quad \ell=1, \ldots,s.
 \end{array}
\end{equation}
and, then, determine, over all $(\varepsilon,j) \in \Xi_r
\times\{1,\ldots,s\}$, the maximum of the corresponding optimal
values in \eqref{def:lin_prog_2}. See \cite{CCS04} for details. In
the following two subsections, we show that the problem in
\eqref{def:lin_prog_2} can be interpreted as a network flow
problem.

\subsubsection{Dual of a minimum cost problem and its properties}
\label{sec:dual}

In this subsection we show that the problem in
\eqref{def:lin_prog_2} is the dual of a minimum cost network flow
problem introduced in subsection \ref{subsec:intro_LP}. To arrive at this conclusion we need to introduce some
additional notation.

We define a directed graph $G=(V,E)$ with the vertex set
$V=\{-N-2,\ldots, 0\}^s$ and the edge set
$$
 E=\{ (u,v)=(\beta, \beta-\epsilon_\ell) \ : \ \beta \in \{-N-1,\ldots, 0\}^s, \ \ell=1,\ldots,s\}.
$$
Note that the undirected graph corresponding to $G$ is connected.
Moreover, $G$ is acyclic, i.e. $G$ does not contain a directed
cycle. For a fixed $r \in \NN$ and for each pair
$(\varepsilon,j) \in \Xi_r \times\{1,\ldots,s\}$, we also define a function $d: E
\rightarrow \RR$ by
$$
 d(e)=d_{uv}= B^{[r]}_{j,\ell}(\varepsilon-M^r\beta), \quad e=(u,v)=(\beta, \beta-\epsilon_\ell) \in E.
$$

\begin{definition}
A function $x: V \rightarrow \mb{R}$, where $x(v)=x_v$ denotes the function value
at a vertex  $v \in V$,   is
called $C-$smooth if
\begin{equation}
 \label{eq:feasible}
   -1 \le x_u - x_v \le 1 \;\;\;\;\mbox{\rm for all $(u,v) \in E$.}
\end{equation}
\end{definition}
We call such a function $C-$smooth to emphasize that the constraints in \eqref{eq:feasible} appear in
convergence analysis of $S_\ab$.

Note that solving \eqref{def:lin_prog_2} for the unknown sequence $\cb$ amounts to solving for $x$ the problems
\begin{equation}
 \label{eq:LPd}
  z^*_C=\max \,\{\sum_{(u,v) \in E} d_{uv} (x_u-x_v):
   \mbox{\rm $x$ is $C$-smooth} \}
\end{equation}
and  then finding the maximum of these values over $\varepsilon
\in \Xi_r$ and  $j=1, \ldots, s$.

We compare next the properties of the optimization problem
\eqref{eq:LPd}  with the properties of the difference subdivision
operator $S_\Bb$.

\begin{remark} \label{rem:2} Define the weights
$$
 b_v^d=\sum_{u:(v,u) \in E} d_{vu}-\sum_{u:(u,v) \in E} d_{uv} \quad \hbox{for} \ v \in V.
$$
The identity
\begin{equation}
 \label{eq:sum0}
  \sum_{v \in V} b^d_v = \sum_{v \in V} \big(\sum_{u:(v,u) \in E}
  d_{vu} -\sum_{u:(u,v) \in E} d_{uv}\big)=0
\end{equation}
is due to the simple fact that each of the terms $d_{uv}$ appears in the above identity twice with the opposite signs.
Note that the identity \eqref{eq:sum0} is equivalent to
$$
 \|S_\Bb \nabla \cb\|_\infty=0 \quad \hbox{for a constant sequence} \ c(\alpha)=c(\beta), \quad \alpha,\beta \in \ZZ^s.
$$
\end{remark}

The property $\displaystyle z^*_C=\|S^r_\Bb|_\nabla\|_\infty \le \|S^r_\Bb\|_\infty=\|d\|_1 = \sum_{(u,v) \in E} |d_{uv}|$ is reflected in the following lemma.

\begin{lemma}
 \label{lem:simple}
  The problem $(\ref{eq:LPd})$ has an optimal solution, so it is
  feasible and not unbounded. Moreover, its optimal value $z_C^*$ satisfies
  \[
       0 \le z_C^* \le \|d\|_1 = \sum_{(u,v) \in E} |d_{uv}|.
  \]
\end{lemma}
\begin{proof}
 The constant function $x=0$ is $C$-smooth, so problem \eqref{eq:LPd}
 is feasible, i.e., has feasible solutions, and  $z^*_C
 \ge 0$. If $x$ is $C$-smooth, then the objective function
 $$
  f^d(x)=\sum_{(u,v) \in E} d_{uv}(x_u-x_v)
 $$
 satisfies
 \begin{eqnarray*}
  f^d(x) &\le& |f^d(x)| \le \sum_{(u,v) \in E} |d_{uv}(x_u-x_v)|\\  &=& \sum_{(u,v) \in E}
 |d_{uv}||x_u-x_v| \le \sum_{(u,v) \in E} |d_{uv}| = \|d\|_1.
 \end{eqnarray*}
\end{proof}

Moreover, we associate to $G$  symmetric directed graph
$\bar{G}=(V,\bar{E})$ with the edge set
$$
  \bar{E}=E \cup \{ (v,u)\ :\ (u,v) \in E\}.
$$
It is easy to see that the linear optimization problem in
\eqref{eq:LPd} is equivalent to
\begin{equation}
 \label{eq:LPd2}
 \hspace{-2cm} \hbox{(DualFlow($d$))} \hspace{2cm}
 \begin{array}{lcl}
   \max &    \displaystyle \sum_{v \in V} b^d_v \,x_v \\
   \mbox{\rm subject to} \\
        & x_v-x_u \le 1 &((u,v) \in \bar{E}).
 \end{array} \notag
\end{equation}

\subsubsection{Minimum cost network flow problem}

In this subsection we cast a more detailed look at network flow problems introduced in subsection 2.2.
The problem in DualFlow($d$) is the standard form of a dual of the
following minimum cost network flow problem
\begin{equation}
 \label{eq:LPp}
 \hspace{-1cm} \hbox{(Flow($d$))} \hspace{2cm}
 \begin{array}{lcl}
   \min &     \displaystyle \sum_{(u,v) \in \bar{E}} f_{uv} \\
   \mbox{\rm subject to} \\
        &   \displaystyle \sum_{u: (v,u) \in \bar{E}} f_{vu}  -
            \sum_{u: (u,v) \in \bar{E}} f_{uv} = b_v^d & (v \in V)
            \\ \\*[\smallskipamount]
        & f_{uv} \ge 0  \quad &((u,v) \in \bar{E}).
 \end{array} \notag
\end{equation}
The flow variable $f_{uv}$ in Flow($d$) represents the flow from $u$ to $v$ along
the edge $(u,v)$. The linear constraints (equations) are flow balance constraints, see (\ref{eq:MCNF}).
For each vertex $v$, these constraints imply that the difference
between total flow out of vertex $v$ and the total flow into the
same vertex is equal to $b^d_v$, which may be considered as the
divergence (net supply) at $v$. A {\em feasible flow} $f$ is a
function $f: \bar{E} \rightarrow \mb{R}$ whose function values
$f((u,v))=f_{uv}$ satisfy the constraints of Flow($d$). An {\em optimal flow} $f^*$ is
a feasible flow that minimizes the objective function $\displaystyle \sum_{(u,v) \in \bar{E}} f_{uv}$ in
Flow($d$).

\begin{remark} \label{optimal_flow_nec_cond}
 If $f^*$ is an optimal flow, then, for each edge $(u,v) \in E$,
 either $f^*_{uv}$ or $f^*_{vu}$ is zero. Otherwise, one could
 reduce both $f^*_{uv}$ and $f^*_{vu}$ by the same small positive
 quantity, which would contradict the optimality of $f^*$.
\end{remark}

The objective function $ \displaystyle \sum_{(u,v) \in \bar{E}}
f_{uv}$ represents the total flow cost. The problem Flow($d$) is a
quite special network flow problem: the costs, i.e. the
coefficients of $f_{uv}$ in the objective function, on the edges
are all 1; and there is no upper bound on the flow in each edge.
These properties together with $\displaystyle{\sum_{v \in V}
b^d_v=0}$ and the fact that the undirected graph associated with
$G$ is connected, imply that the problem Flow($d$) is feasible and
has an optimal solution, see \cite{Vanderbei08}. The existence of
an optimal solution of Flow($d$) also follows from Lemma
\ref{lem:simple} as the dual problem DualFlow($d$) has an optimal
solution. The following result is one of the consequences of this
duality relationship.

\begin{theorem}
 \label{thm:LPduality}
  The optimal value $z^*_C$ in \eqref{eq:LPd} equals the
  optimal value in the network flow problem Flow($d$).
\end{theorem}
\begin{proof}
 This follows from the network flow duality theory, see e.g. \cite{AMO93}.
\end{proof}

\subsection{Vector case}

In the vector case, i.e. $n > 1$, the mask $\Ab \in \ell^{n \times n}(\ZZ^s)$ has matrix
entries $A(\alpha)$, $\alpha \in \ZZ^s$.
The associated first difference scheme is given by repeated applications of the operator $S_\Bb: \ell^{ns}(\ZZ^s) \rightarrow \ell^{ns}(\ZZ^s)$.
There is no conceptual change in the structure of the linear optimization problems in \eqref{def:lin_prog_2}, see \cite{CCS04} for details.
Therefore, even in the vector case, the convergence analysis of subdivision schemes profits from the theory of
network flows. We omit the formulations of the corresponding results to avoid repetitions.


\section{Optimal first difference masks}
\label{sec:optimal_1}


In this section we show that there exists an optimal difference
mask $\Bb^* \in \ell_0^{s \times s}(\ZZ^s)$, possibly different
for each $r \in \NN$, such that the corresponding operator
$S_{\Bb^*}$ in \eqref{def:SB} satisfies $\nabla S^r_\ab=S_{\Bb^*}
\nabla$ and
\begin{equation} \label{prop:optimal_1_mask}
 \|S_{\Bb^*}\|_\infty=\|S_{\Bb^*}|_\nabla\|_\infty=\|S^r_{\Bb}|_\nabla\|_\infty
\end{equation}
for any other $S_{\Bb}$ satisfying  $\nabla S^r_\ab=S^r_{\Bb}
\nabla$. The algorithm for construction of $\Bb^*$ in section
\ref{section:algorithms_and_examples} is such that for a given
difference mask $\Bb$ with rational entries the optimal mask is also rational.
Thus, the norm $\|S_{\Bb^*}\|_\infty$ is rational, which allows
for analytic arguments in convergence proofs for $S_\ab$. In the
multivariate case, such masks $\Bb^*$ possibly differ for each $r
\in \NN$, see Example \ref{ex:iterated_mask}.

 \subsection{Univariate case} \label{subsec:univariate_case}

  In the univariate case, it is well known that the operator $S_\Bb$ is unique and the maximizing sequence in \eqref{def:lin_prog_2} is
  determined uniquely, up to a constant sequence, by
  $$
    c(\beta)-c(\beta-1) =\hbox{sgn} B^{[r]}(\varepsilon-M^r\beta), \quad \beta \in \{-N-1,\ldots,0\},
  $$
  i.e. $z_C^*=\|S^r_\Bb|_\nabla\|_\infty=\|S^r_\Bb\|_\infty=\|d\|_1$ for any $r \in \NN$. The same holds for higher order difference
  schemes. This property of $z_C^*$  also follows directly from Theorem \ref{thm:theright-d}
  in section \ref{subsection:optimal_mask_1_multi}
  and  Theorem \ref{th:existence_optimal_mask_general_case} in section
  \ref{section:optimal_mask_2_multi}.

\subsection{Multivariate case} \label{subsection:optimal_mask_1_multi}

In the multivariate case, the difference subdivision operator $S_\Bb$ in
\eqref{def:S_AvsS_B} is not unique, see \cite{CCS04}.

Fix $r \in \NN$, $\varepsilon \in \Xi_r$ and $j=1, \ldots, s$. The
next result shows that there exists $\Bb^* \in \ell_0^{s \times
s}(\ZZ^s)$ such that
$$
  z^*_{C}=\sum_{\beta \in K} \sum_{\ell=1}^s |B^*_{j,\ell}(\varepsilon-M^r\beta)|=\sum_{(u,v) \in \bar{E}} f^*_{uv},
$$
i.e., the restricted and non-restricted norms of $S_{\Bb^*}$ coincide.

 \begin{theorem}
 \label{thm:theright-d}
  Let $d: E \rightarrow \RR$ be given. Let $f^*$ be an optimal flow in \textup{Flow($d$)} and let $x^*$ be an
  optimal solution of \textup{DualFlow($d$)}. Define the function $d^*: E \rightarrow \RR$ by
  $d^*_{uv}=f^*_{uv}-f^*_{vu}$ for each $(u,v) \in E$.

  Then  $b^d_v=b^{d^*}_v$ for all $v \in V$, i.e. \textup{Flow($d$)} and \textup{Flow($d^*$)} coincide, and
  so do \textup{DualFlow($d$)} and \textup{DualFlow($d^*$)}. Moreover, the common optimal value of these problems is
  equal to $\|d^*\|_1$, i.e.
  \[
      \sum_{(u,v) \in \bar{E}} f^*_{uv}=\sum_{v \in V} b^{d^*}_v x^*_v = \|d^*\|_1.
  \]
 \end{theorem}

\begin{proof}
 Let $x^*$ be an optimal solution of \textup{Dualflow($d$)}.  Define $d^*$ as stated in
 the theorem. Then, as $f^*$ is a feasible solution of  \textup{Flow($d$)},
 for each $v \in V$, we have
\begin{eqnarray*}
     b^d_v &=& 
          \sum_{u:(v,u) \in \bar{E}} f^*_{vu}-\sum_{u:(u,v) \in \bar{E}} f^*_{uv}
                \\*[\smallskipamount]
           &=& \sum_{u:(v,u) \in E} (f^*_{vu}-f^*_{uv}) -\sum_{u:(u,v) \in E} (f^*_{uv}-f^*_{vu})
                          \\*[\smallskipamount]
            &=& \sum_{u:(v,u) \in E} d^*_{vu}-\sum_{u:(u,v) \in E} d^*_{uv}= b^{d^*}_v
  \end{eqnarray*}
 which proves the first statement.

 Next, as $f^*$ is optimal, by Remark \ref{optimal_flow_nec_cond}, for each $(u,v) \in E$
 at most one of the two variables $f^*_{uv}$ and $f^*_{vu}$ can be positive.
 So for each $(u,v) \in E$
 \[
    f^*_{uv}+f^*_{vu}=|f^*_{uv}-f^*_{vu}|=|d^*_{uv}|
 \]
 and therefore, since \textup{Flow($d$)} and  \textup{Dualflow($d$)} have the same optimal value
 \[
   \sum_{v \in V} b^{d^*}_v x^*_v = \sum_{(u,v) \in \bar{E}} f^*_{uv}=\sum_{(u,v) \in E} (f^*_{uv}+f^*_{vu})=\sum_{(u,v) \in E} |d^*_{uv}|=\|d^*\|_1
 \]
 as desired.
\end{proof}

We are guaranteed to have an integral optimal $d^*$, i.e. a rational optimal $\Bb^*$, under the additional assumption that $d$ is integral.

\begin{corollary}
  Let $d: E \rightarrow \ZZ$ and $f^*$ be an integral optimal flow in \textup{Flow($d$)}. Then  $d^*: E \rightarrow \ZZ$
  defined by $d^*_{uv}=f^*_{uv}-f^*_{vu}$, $(u,v) \in E$, satisfies
  \[
      \sum_{(u,v) \in \bar{E}} f^*_{uv} = \|d^*\|_1.
  \]
\end{corollary}
\begin{proof}
 The network flow theory (see Subsection 2.2 or \cite{AMO93}) guarantees the existence of an {\em integral} optimal solution $f^*$
 of Flow($d$), if $d$ is integral. The claim follows then from Theorem \ref{thm:theright-d}.
\end{proof}


\begin{remark} It can happen that $\|d^*\|_1 <\|d\|_1$. For example, consider a directed graph with vertices
$V=\{ u=(0,0), v=(0,1), w=(1,1), z=(1,0)\}$ and let
$d_{uv}=d_{vw}=1$ and $d_{zw}=-1$. Then $b^d_u=1$ , $b^d_z=-1$ and
$b^d_v=b^d_w=0$. An optimal flow $f^*$ is given by $f^*_{uz}=1$
and zero otherwise, so $d^*_{uz}=1$ and zero otherwise. One of the
optimal dual solutions is $x^*_u=0$, $x^*_v=x^*_z=-1$ and
$x^*_w=-2$. The common optimal value is  1 (recall that we have
taken the negative in the dual), and so $\|d^*\|_1=1$ while
$\|d\|_1=3$.
\end{remark}

We state next some necessary and sufficient conditions for $d=d^*$. These sufficient conditions are easy to check and, if satisfied,
yield $d^*$ without solving DualFlow($d$).

Recall that an edge $e \in \bar{E}$ has the form $e=(u,v)$ and
that, for $u=(u_1, u_2, \ldots, u_s)$, $v=(v_1, v_2, \ldots,
v_s)$, we have $| v_k-u_k |=1$ for some $k$ and $v_j=u_j$ for all
$j \not = k$. We say that $e$ is a {\em $k$-positive} edge, if
$v_k=u_k+1$, while if $v_k=u_k-1$ we call $e$ a {\em $k$-negative}
edge. A path $P$ in $\bar{G}$ is called {\em monotone} if, for a
fixed $k$, it either contains only $k$-positive edges or only
$k$-negative edges. The path $P$ in the support of $f^*$ consists
of the edge set $\{(u,v) \in P: f^*_{uv}>0\}$.

\begin{theorem}
 \label{thm:mon-path}
  Let $f^*$ be an optimal solution in \textup{Flow($d$)} for a graph $\bar{G}$.
  Then each path $P$ in the support of $f^*$ is monotone.
 \end{theorem}
\begin{proof}
 Assume that the support of $f^*$ contains a path $P$ which is not monotone. Say that
 $P$ has $m$ edges and that its vertices (in that order) are
 $u^0, u^1, \ldots, u^m \in V$. We represent $P$ by a $(0,1,-1)$-matrix $A$
 of size $m \times s$ whose $i-$th row is $u^i-u^{i-1}$. The definition of $E$
 assures that each entry in $A$ is either $0$, $1$ or $-1$.  Since $P$
is not monotone, $A$ contains a column with both $1$ and $-1$. We
choose such a column $k$ for which the rows $i_1$ and $i_2$
containing $1$ and $-1$ are such that $|i_1-i_2|$ is minimal; we
may assume $i_1<i_2$. Note that in these rows $i_1$ and $i_2$ the
only non-zero entries are in the $k-$th column.

Let the matrix $A'$ be obtained from $A$ by deleting rows $i_1$ and $i_2$. Then
$A'$ corresponds to a new path $P'$ having the same end vertices as $P$.
We may define a new flow $f'$ accordingly by replacing a flow of one unit
along $P$ by a flow of one unit along $P'$. Then, since the symmetric difference
between $P$ and $P'$ is a cycle, $f'$ satisfies the flow balance constraints. Moreover,
$\sum f'_{uv}= \sum f^*_{uv}-2$, contradicting that $f^*$ was optimal. Thus we have
shown that the support of $f^*$ only contains paths that are monotone.
\end{proof}

For $i=1,\ldots, s$ define $E_i=\{ (u,u-\epsilon_i) \in E \}$.
Then $\{E_i\ : \ i=1,\ldots,s\}$ is a partition of the edge set
$E$, i.e. the sets $E_i$ are disjoint and $E$ is the union of
$E_i$.

\begin{theorem}
 \label{thm:signed-d}

  Consider a graph $\bar{G}$. Let $\kappa_i \in \{-1,1\}$ for $1 \le i \le s$, and assume
  that $\sign(d_{uv}) \in \{0,\kappa_i\}$ for all $(u,v) \in E_i$, $1\le i \le s$.

 \begin{description}
  \item[$(i)$] The flow $f^*$ with
  $$
   (f^*_{uv}, f^*_{vu})=\left\{
                          \begin{array}{ll}
                            (d_{uv},\ 0), & \hbox{if } \kappa_i=1, \\
                            (0,-d_{uv}), & \hbox{if } \kappa_i=-1
                          \end{array}
                        \right. \  \hbox{for all } (u,v) \in E_i, \  i=1,\ldots,s,
  $$
  is  optimal for \textup{Flow($d$)} with optimal value $\|d\|_1$.
  \item[$(ii)$] The function $x^*: V \rightarrow \RR$ with $x^*(v)=-\sum_{i=1}^s \kappa_i v_i$ for $v=(v_1,v_2, \ldots, v_s) \in V$
  is optimal in \textup{DualFlow($d$)} with optimal value $\|d\|_1$.
\end{description}
 \end{theorem}
\begin{proof} To prove $(i)$ and $(ii)$ we use the standard technique,
based on weak duality (see Subsection 2.2 and  \cite{Vanderbei08}). It suffices
to show that $f^*$ and $x^*$ defined in $(i)$ and $(ii)$,
respectively, are feasible and
$$
 \sum_{v \in V} b^{d}_v x^*(v)=\sum_{(u,v) \in \bar{E}} f^*_{uv}.
$$
Let $i=1,\dots,s$ and $e=(u,v) \in E_i$. Then
 \[
       x^*(u)-x^*(v)=x^*(u)-x^*(u-\epsilon_i)=\kappa_i u_i -\kappa_i (u_i -1)=\kappa_i.
 \]
This shows that $|x^*(u)-x^*(v)|=1$ for each edge $(u,v) \in
\bar{E}$, so $x^*$ is feasible in \textup{DualFlow($d$)}.
 Clearly $f^*$ is feasible in \textup{Flow($d$)} as $f^*_{uv}\ge 0$ and
 its divergence in a vertex
 $v \in V$ equals the divergence of $d$ in $v$ which is $b^d_v$.
 Moreover, we have
  \begin{eqnarray*}
     \sum_{v \in V} b^{d}_v x^*(v) &=&
         \sum_{v \in V} (\sum_{u:(v,u) \in E} d_{vu}-\sum_{u:(u,v) \in E} d_{uv})x^*(v)  \\*[\smallskipamount]
           &=& \sum_{(u,v) \in E} d_{uv} (x^*(u) - x^*(v)) \\*[\smallskipamount]
           &=& \sum_{i=1}^s \sum_{(u,v) \in E_i} d_{uv} (x^*(u) - x^*(v)) \\*[\smallskipamount]
           &=& \sum_{i=1}^s \sum_{(u,v) \in E_i} d_{uv} \kappa_i = \sum_{i=1}^s \sum_{(u,v) \in E_i} |d_{uv}| \\*[\smallskipamount]
           &=& \sum_{i=1}^s \sum_{(u,v) \in E_i} (f^*_{uv}+ f^*_{vu}) = \sum_{(u,v) \in \bar{E}} f^*_{uv}
  \end{eqnarray*}
 Therefore, by duality (Theorem \ref{thm:LPduality}) it follows that $f^*$ is optimal in \textup{Flow($d$)}, $x^*$ is optimal in \textup{DualFlow($d$)} and, finally, that the optimal value equals $\|d\|_1$.
 \end{proof}

 \begin{corollary}
 \label{cor:nonneg-d}
  Consider a graph $\bar{G}$, and let $d$ be nonnegative. Define $f^*_{uv}=d_{uv}$
  and $f^*_{vu}=0$ for each $(u,v) \in E$ and $x^*(v)=-\sum_{i=1}^s v_i $  for $v=(v_1,v_2, \ldots, v_s) \in V$.
Then $f^*$ is optimal in   \textup{Flow($d$)}, $x^*$ is optimal in
\textup{DualFlow($d$)}
  and the optimal value equals $\|d\|_1$.
 \end{corollary}
 \begin{proof}
  Let $\kappa_i=1$ for each $i \le s$ and apply Theorem \ref{thm:signed-d}.
 \end{proof}

We conclude this subsection with an example of an optimal first difference mask $\Bb^*$
that satisfies
$$
 \| S_{\Bb^*}\|_\infty = \|S_{\Bb^*} |_\nabla\|_\infty,
$$
but does not satisfy
$$
 \| S^2_{\Bb^*}\|_\infty \not= \|S^2_{\Bb^*} |_\nabla\|_\infty.
$$

\begin{example} \label{ex:iterated_mask}
 Let $M=2I$. Consider a bivariate scalar subdivision scheme with the mask
 $$
 \ab=
      \begin{array}{ccccccc}
             & \vdots & \vdots & \vdots & \vdots & \vdots &\\
       \dots & 0      & 0      & \frac{1}{4}      & \frac{1}{2}      & \frac{1}{4} & \dots   \\ \\
       \dots & 0      & \frac{1}{2}      & 1      & \frac{1}{2}      & 0 & \dots\\ \\
       \dots & {\bf \frac{1}{4}}      & \frac{1}{2}      & \frac{1}{4}      & 0      & 0 & \dots \\
             & \vdots & \vdots & \vdots & \vdots & \vdots &
      \end{array}
 $$
 supported on $\{0,\ldots,3\}^2$. The bold entry corresponds to the index $(0,0)$. In this case, the directed graph $G=(V,E)$ has
 vertices $V=\{-5,\ldots,0\}^2$ and the edge set $E=\{(\beta, \beta-\epsilon_\ell)\ : \ \beta \in \{-4,\ldots,0\}^2, \ \ell=1,2\}$. The nonzero
part of the optimal mask $\Bb^* \in \ell_0^{2 \times 2}(\ZZ^2)$
for the first difference scheme is
$$
 \Bb^*= \frac{1}{4}
      \begin{array}{cccccc}
         & \left(\begin{array}{cc}0&0\\0&0 \end{array}\right) &
               \left(\begin{array}{cc}0&0\\0&0 \end{array}\right) &
               \left(\begin{array}{cc}0&0\\ 1 &0  \end{array}\right)&
               \left(\begin{array}{cc}0&0\\0&0  \end{array}\right)&   \\
         & \left(\begin{array}{cc}0&0\\1&0 \end{array}\right) &
               \left(\begin{array}{cc}0&0\\0&0 \end{array}\right) &
               \left(\begin{array}{cc}1&0\\-1&0  \end{array}\right)&
               \left(\begin{array}{cc}1&0\\0&0 \end{array}\right)&    \\
         & \left(\begin{array}{cc}0&0\\-1&0 \end{array}\right) &
               \left(\begin{array}{cc}2&0\\0&0 \end{array}\right) &
               \left(\begin{array}{cc}2&0\\0&2  \end{array}\right)&
               \left(\begin{array}{cc}0&0\\0&2  \end{array}\right)&   \\
        & {\bf \left(\begin{array}{cc}{\bf 1}&{\bf 0}\\{\bf 0}&{\bf 1} \end{array}\right)} &
               \left(\begin{array}{cc}1&0\\0&2 \end{array}\right) &
               \left(\begin{array}{cc}0&0\\0&1  \end{array}\right)&
               \left(\begin{array}{cc}0&0\\0&0  \end{array}\right)&    \\
      \end{array}
$$
The bold entry again corresponds to the index $(0,0)$.
The nonzero entries of its second iterated mask $(B^*)^{[2]}$ for the
coset $\varepsilon=(2,1)$  are
\begin{eqnarray*}
   (B^*)^{[2]}(2,5) &=& \left(\begin{array}{cc}0&0\\ \frac{1}{8}&0 \end{array}\right), \quad
   (B^*)^{[2]}(6,5) = \left(\begin{array}{cc} \frac{1}{8}&0\\- \frac{1}{16}&0 \end{array}\right)  \\
   (B^*)^{[2]}(2,1) &=& \left(\begin{array}{cc} \frac{1}{8}&0\\- \frac{1}{16}& \frac{1}{8} \end{array}\right) \quad \hbox{and} \quad
   (B^*)^{[2]}(6,1) = \left(\begin{array}{cc}0&0\\0& \frac{1}{8} \end{array}\right).
\end{eqnarray*}
On the contrary, the corresponding linear optimization problem for $\varepsilon=(2,1)$ and $j=2$ in \eqref{eq:LPd}
 and with nonzero values
\begin{eqnarray*}
 d_{(0,0)(-1,0)}&=&d_{(-1,-1)(-2,-1)}=-\frac{1}{16}, \\
 d_{(0,0)(0,-1)}&=& d_{(-1,0)(-1,-1)}= d_{(0,-1)(-1,-1)}=\frac{1}{8},
\end{eqnarray*}
yields $ z_C^*=\frac{6}{16}$. And, indeed, the corresponding nonzero entries of the optimal
second iterated mask constructed from $d^*$ for this coset are
$$
      \begin{array}{ccccc}
             & \vdots & &\vdots & \\
        \dots & \left(\begin{array}{cc}0&0\\ \frac{1}{16}&0 \end{array}\right) & \dots&
               \left(\begin{array}{cc} \frac{1}{8}&0\\- \frac{1}{16}&0 \end{array}\right) & \dots   \\
           & \vdots & & \vdots & \\
        \dots & \left(\begin{array}{cc} \frac{1}{8}&0\\0& \frac{1}{16} \end{array}\right) & \dots&
               \left(\begin{array}{cc}0&0\\0& \frac{3}{16} \end{array}\right) & \dots   \\
      & \vdots & & \vdots & \\
      \end{array}
$$
\end{example}

\subsection{Flow algorithm} \label{section:algorithms_and_examples}

In this section we present the  {\em successive
shortest path algorithm} (\cite{AMO93}) for our minimum cost network flow problems.  This algorithm determines the optimal flow $f^*$
in \textup{Flow($d$)}, which defines an optimal mask of the first
difference schemes as stated in Theorem \ref{thm:theright-d}. An
advantage of using this particular algorithm is that it guarantees
an integral solution, if the input is integral. For a given flow
$f$, the algorithm defines the so-called {\em residual network}
$G(f)$ which consists of
\begin{itemize}
\item[-] the given vertices $V$ of the graph $\bar{G}$,
\item[-] all edges $e =(u,v)\in \bar{E}$ and their copies, called "backward" edges
$e^{-1}$ whose direction is opposite to $e$. These edges are created only for the edges 
$e \in \bar{E}$ that consitute the shortests 
paths determined by  Dijkstra's algorithm in step 2.2. of the Flow algorith given below.
\end{itemize}
In the residual network $G(f)$ each edge is assigned a capacity
$r_{(u,v)}=f_{uv}$ for $e =(u,v)\in \bar{E}$ and for backward
edges
$$
 r_{(u,v)^{-1}}=
  \begin{cases}
    f_{uv} & \text{if}, f_{uv} >0, \\
    0 & \text{otherwise}.
  \end{cases}
$$
For a given function $\pi: V \rightarrow \mb{Z}$, each edge in
$G(f)$ is also assigned the so-called {\em reduced cost}
\[
     c^{\pi}_{(u,v)}=1-\pi(u)+\pi(v), \quad (u,v) \in \bar{E},
\]
and for the backwards edges $(u,v)^{-1}$, $(u,v) \in \bar{E}$,
\[
   c^{\pi}_{(u,v)^{-1}}=-1-\pi(u)+\pi(v).
\]
The purpose of introducing the backward edges is that they allow
us to decrease the flow through the original edges in $\bar{E}$ by
sending it along the corresponding backward edges.

The algorithm starts with the zero flow $f=0$ and performs a
finite number of iterations consisting of adding the flow along
the shortest path between the end vertices of an edge $(u,v)$
where $d_{uv}$ is nonzero. The shortest path is computed in the
residual network $G(f)$ using the reduced costs $c^{\pi}$. The
function $\pi$ is introduced to ensure that the reduced costs stay
nonnegative, which makes the shortest path calculation more
efficient, as we can use Dijkstra's algorithm. For further
explanation and details on the successive shortest path algorithm,
see \cite{AMO93}.

\vspace{0.3cm}

{\scriptsize

\noindent {\bf Flow algorithm:}

\vspace{0.2cm}
\noindent {\bf Input: } a function $d: E \rightarrow \mb{Z}$.

\begin{itemize}
\item[1.]  {\bf Initial step:}  Compute $b^d_v$ for $v \in V$. Set
$\epsilon(v):=b^d_v$ for  $v \in V$. Define $\mc{E}_+=\{v \in V:
\epsilon(v)>0\}$ and $\mc{E}_-=\{v \in V: \epsilon(v)<0\}$.
Initialize the flow $f:=0$, define the residual network $G(f)$,
and set $\pi(v)=0$ for each $v \in V$.

\item[2.]  {\bf While} $\mc{E}_+ \not = \emptyset$ {\bf do}
 \begin{enumerate}
   \item[2.1]  Choose any $v_+ \in \mc{E}_+$ and $v_- \in \mc{E}_-$.
   \item[2.2]  Dijkstra's algorithm: uses as edge lengths the reduced costs
   $c^{\pi}_{(u,v)}=1-\pi(u)+\pi(v)$ to compute the shortest path distances
   $\delta(v)$ from $v_+$ to each other vertex $v \in V$;
   determines the shortest path $P$ from $v_+$ to $v_-$.
   \item[2.3] Update $\pi$ by $\pi:= \pi-\delta$, compute
   $\gamma:=\min\{\epsilon(v_+), -\epsilon(v_-), \min\{r_{ij}: (i,j) \in P\} \}$ and
   augment $f$ by adding a flow of value $\gamma$ along the path $P$.
   Update the residual network $G(f)$, i.e. update the reduced
   costs; set $\epsilon(v_+)=\epsilon(v_+)-\gamma$,
   $\epsilon(v_-)=\epsilon(v_-)+\gamma$; update $\mc{E}_+$,
   $\mc{E}_-$.
 \end{enumerate}
 \end{itemize}
 \noindent  {\bf Output:} optimal flow $f^*$ and optimal dual variable $x:=-\pi$.
}

\begin{theorem}
 \label{thm:flow-alg-thm}
  The flow algorithm  solves both  \textup{Flow($d$)} and
  the dual problem  \textup{Dualflow($d$)}. Its complexity is $O(Bn^2)$ where
  $B=(1/2)\sum_v |b^d_v|$ is an upper bound on the number of iterations,
  and $O(n^2)$ is the complexity of Dijkstra's algorithm for solving the shortest path problem with nonnegative costs in a graph with $n$  vertices.
\end{theorem}
\begin{proof}
 The correctness of the general algorithm is shown in \cite{AMO93}.  The complexity statement follows from the
 fact that in each iteration, by integrality of $d$, the flow is augmented by a
 positive integer.
\end{proof}

The next simple example illustrates the flow algorithm. It also
stresses the necessity of using the algorithm for finding the
optimal flow $f^*$ even for seemingly simple examples of our very
special network flow problems.

\begin{example}  Consider the graph $\bar{G}$ with $s=2$ and assume
$d$ is such that $b^d_v=1$ when $v \in \{(0,0),(-2,-2)\}$,
$b^d_v=-1$ when $v \in \{(-1,-1),(-3,-3)\}$, and $b^d_v=0$ otherwise.
Initially, in the flow algorithm,
$$
 f=0, \quad \mc{E}_+=\{(0,0),(-2,-2)\}, \quad
 \mc{E}_-=\{(-1,-1),(-3,-3)\},
$$
and we (may) choose $v_+=(-2,-2)$ and
$v_-=(-1,-1)$. A shortest path $P$ from $v_+$ to $v_-$ consists of the nodes
$(-2,-2)$, $(-1,-2)$, $(-1,-1)$ and it has cost 2. Note that the updated
edge cost for each  edge in $P$ is $-1$. In the next, and final,
iteration, we (must) choose $v_+=(0,0)$ and $v_-=(-3,-3)$ the shortest path
$P'$ consists of the nodes
$(0,0)$, $(0,-1)$, $(-1,-1)$, $(-1,-2)$, $(-2,-2)$, $(-2,-3)$,
$(-3,-3)$. As a result, the flow cancels out on $P$, and we have an
optimal flow with $f=1$ on four edges so the optimal value is $4$.

This example shows that the obvious heuristic method of
successively adding a shortest path, while maintaining previous paths,
may go wrong. Doing so  we would get a solution with two paths, one
of length $2$ and the other of length $6$, so a total cost of $8$, while
the optimal value is $4$.
\end{example}


\section{Optimal higher order difference schemes} \label{section:optimal_mask_2_multi}


In this section we investigate the existence of optimal masks for
higher order difference schemes used for studying the regularity
of subdivision in the scalar case, i.e. $n=1$. The vector case is
more technical, but the computation of the restricted norms we
consider here and the derivation of the optimal difference schemes
are conceptually very similar to what we do in the scalar case.
The unconvinced reader is referred to \cite{C11} and Example
\ref{ex:vector}.

The smoothness analysis of $S_\ab$ is based on the spectral properties of the higher order difference schemes
$S_{\Bb_k}$, $k \ge 2$, derived from $S_\ab$. In our notation, we have
$\Bb_1=\Bb$, where $\Bb$ is the first difference scheme from section \ref{sec:1_diff_scalar}. The $k-th$ order backward difference operator
$\nabla^k : \ell(\ZZ^s) \rightarrow \ell^{N_{s,k}}(\ZZ^s)$,
$N_{s,k}=\left( \begin{array}{c} s+k-1 \\ s-1 \end{array}\right)$, is defined by
\begin{equation}\label{def:nabla_k}
 \nabla^k=\left( \nabla_1^{\mu_1} \dots \nabla_s^{\mu_s}\right)_{\mu=(\mu_1, \ldots, \mu_s) \in \NN_0^s \atop |\mu|=k},
\end{equation}
where, for $\ell =1, \ldots, s$,
$$
 \nabla_\ell^{\mu_\ell}=\nabla_\ell \nabla_\ell^{\mu_\ell-1}, \quad \mu_\ell \in \NN, \quad \nabla_\ell^0=\hbox{id}.
$$
The $k-th$ order difference schemes $S_{\Bb_k}$ satisfy
\begin{equation}\label{id:Sa=Sbk}
 \nabla^k S^r_\ab= S_{\Bb_k}^r \nabla^k, \quad r \in \NN.
\end{equation}
We denote the entries of the matrices $B_k^{[r]}(\alpha)$ by
$B_{k,j,\mu}^{[r]}(\alpha)$, $j=1, \dots, N_{s,k}$, and
$\mu=(\mu_1, \dots, \mu_s)$ matches the ordering of
$\nabla_1^{\mu_1} \dots \nabla_s^{\mu_s}$ in $\nabla^k$. By
\cite{C11}, the study of the spectral properties of
$S_{\Bb_k}|_{\nabla^k}$ leads to computation of the restricted
norms
$$
 \|S^r_{\Bb_k}|_{\nabla^k}\|_\infty=\max_{\|\nabla^k \cb |_K\|_\infty=1} \max_{\varepsilon \in \Xi_r}
 \left\| \sum_{\beta \in K} B^{[r]}_k(\varepsilon -M^r \beta) (\nabla^k \cb)(\beta) \right\|_\infty,
$$
where $K=\{-N-k, \ldots,0\}^s$. Let $r \in \NN$,
$j=1,\dots,N_{s,k}$ and $\varepsilon \in \Xi$. See \cite{C11} for details. The linear
constraints $ \|\left(\nabla^k \cb  \right) (\beta) \|_\infty \le
1$ for $\beta \in K$ do not allow us to interpret the linear
optimization problem
\begin{equation} \label{def:lin_prog_3}
 \begin{array}{lcl}
   \max &  \displaystyle \sum_{\beta \in K} \sum_{\mu \in \NN_0^s \atop |\mu|=k} B_{k,j,\mu}^{[r]}(\varepsilon-M^r\beta) (\nabla_1^{\mu_1}\ldots
  \nabla_s^{\mu_s} \cb)(\beta) \\
   \mbox{\rm subject to} \\
        & -1 \le \left(\nabla_1^{\mu_1} \ldots  \nabla_s^{\mu_s} \cb \right)(\beta) \le 1, \quad \beta \in K, \quad \mu \in \NN_0^s, \quad |\mu|=k,
 \end{array}
\end{equation}
as a network flow problem, compare with \eqref{eq:MCNF}. However, we can still show that for
each $r \in \NN$ there exists an optimal mask $\Bb^*$ such that
$$
 \| S^r_{\Bb_k}|_{\nabla^k}\|_\infty= \| S_{\Bb^*}|_{\nabla^k}\|_\infty=\| S_{\Bb^*}\|_\infty
$$
and
$$
 \nabla^k S^r_{\Ab}= S^r_{\Bb_k} \nabla^k = S_{\Bb^*} \nabla^k.
$$
Denote by $\bone$ and by $\bnull$  vectors of all ones and all
zeros, respectively. The problem in \eqref{def:lin_prog_3} is
equivalent to
\begin{equation}\label{eq:LPd_k}
 z^*=\max\{ d^T (\Delta x) \ : \ -\bone \le \Delta x \le \bone\}
\end{equation}
with appropriately defined vector $d \in \RR^{|N+k|^s}$ of the corresponding entries of $B_{k,j,\mu}^{[r]}(\varepsilon-M^r\beta)$ in the objective function and the matrix $\Delta$ reflecting the linear constraints in \eqref{def:lin_prog_3}.

\begin{theorem} \label{th:existence_optimal_mask_general_case} There exists
a vector $d^* \in \RR^{N_{s,k}(N+k+1)^s}$ such that
the solution of \eqref{eq:LPd_k} satisfies
$$
 z^*=\|d^*\|_1.
$$
\end{theorem}
\begin{proof}
 Note that
 $$
  z^*=\max\{ \left(d^T \Delta \right) x \ : \  \left( \begin{array}{rr} \Delta \\ -\Delta \end{array}\right) x \le
  \left( \begin{array}{rr} \bone \\ -\bone \end{array}\right)\}.
 $$
 By duality \cite{Schrijver}, we get
\begin{eqnarray*}
 z^*&=&\min \{ \bone^T (w+y) \ : \ \left( \begin{array}{rr} w^T &y^T \end{array}\right)  \left( \begin{array}{rr} \Delta \\ -\Delta
\end{array} \right) =d^T \Delta, \ w,y \ge \bnull \}\\
 &=&\min \{ \bone^T (w+y) \ : \ ( d+y-w)^T \Delta =\bnull,\  w,y \ge \bnull \}.
\end{eqnarray*}
Moreover, due to the fact that the supports of the optimal $w$ and $y$ are disjoint, we obtain
\begin{eqnarray}\label{new}
 z^*&=&\min \{ \bone^T |g| \ : \ ( d+g)^T \Delta =\bnull,\ g \in \RR^{m} \} \notag \\
    &=&\min \{ \|d-g\|_1 \ : \ g^T \Delta =\bnull,\ g \in \RR^{m} \}.
\end{eqnarray}
Define $d^*=d-g$.
\end{proof}

Note that the entries of $d^*$ define the elements of  the optimal
mask $\Bb^*$ for the corresponding $r\in\NN$, $j=1,\dots, N_{s,k}$
and $\varepsilon \in \Xi$.

We would like to emphasize that the value $z^*$ coincides with the
one determined by solving \eqref{def:lin_prog_3}, which was
already studied in \cite{C11} for $k \ge 1$ and in \cite{CCS04} for $k=1$. The
equivalent formulation of \eqref{def:lin_prog_3} in \eqref{new}
allows us not only to show the existence of optimal masks, but
also yields its new, very intuitive interpretation: {\it geometrically, $z^*$ is
the distance from $d$ to the nullspace of $\Delta^T$ in the
$\|\cdot \|_1$ norm}.

\section{Examples}
\label{sec:examples}

In this section we illustrate our results with several examples.
We give optimal masks for first and second difference schemes only
for simplicity of presentation. The higher order difference masks
can be determined analogously.

\begin{example} Let $M=\left( \begin{array}{cc}
2&1\\0&2\end{array}\right)$. In this case $\Xi \simeq
\{(0,0),(1,0),(1,1),(1,2)\}$. Consider a scalar bivariate
subdivision mask $\ab$ given in terms of the associated symbol
$$
 a(z)=\sum_{\alpha \in \ZZ^2} a(\alpha)z^\alpha =\frac{1}{4}b^2(z),
 \quad b(z)=\sum_{\varepsilon \in \Xi} z^\varepsilon, \quad
 z \in (\CC \setminus \{0\})^2.
$$
The optimization
problem in \eqref{new}
implemented in Matlab yields an optimal mask $\Bb^* \in \ell_0^{2
\times 2}(\ZZ^2)$ given in terms of the associated matrix symbol
$$
 B^*(z)= \frac{1}{4}\left( \begin{array}{cc} b_{11}(z) & b_{12}(z) \\ b_{21}(z) & b_{22}(z)
 \end{array} \right)
$$
with
\begin{eqnarray*}
 b_{11}(z)&=&(1+z_1)(1+2z_1z_2+z_1^2z_2^2), \quad b_{12}(z)=0, \\
 b_{21}(z)&=&-591/1739-z_2-709/1074z_1z_2-709/1074z_1z_2^2-z_1^2z_2^2-591/1739z_1^2z_2^3,
 \\
 b_{22}(z)&=&1439/1074+2z_1+709/1074z_1^2+709/1074z_1z_2+2z_1^2z_2+1439/1074z_1^3z_2.
\end{eqnarray*}
Note that there is also an optimal mask with integral $b_{ij}(z)$
given by
\begin{eqnarray*}
b_{11}(z)&=&(1+z_1)(1+2z_1z_2+z_1^2z_2^2), \quad b_{12}(z)=0, \\
b_{21}(z)&=&-z_2-z_1z_2^2-z_1^2z_2^2-z_1^2z_2^3 \quad \hbox{and}
\quad b_{22}(z)=1+2z_1+z_1^2+2z_1^2z_2+2z_1^3z_2.
\end{eqnarray*}
If we start the flow algorithm from section \ref{section:algorithms_and_examples}
with $d$ derived from this optimal $\Bb^*$, we get $d^*=d$ as an output.
For optimal masks we get
$$
 \|S_{\Bb^*}\|_\infty=\|S_{\Bb^*}|_\nabla\|_\infty=\frac{3}{4},
$$
which implies convergence of $S_\ab$, i.e. continuity of its
limits.
\end{example}

The next example is of a vector bivariate subdivision scheme
introduced in \cite{CJ01}. The corresponding dilation matrix is
$M=2I$ and $\Xi \simeq
\{0,1\}^2$.

\begin{example} \label{ex:vector} We transform the mask in \cite{CJ01} following the steps in \cite[Example 5.2]{C11} and
obtain
$$
 \Ab = \frac{1}{8}
\begin{array}{cccc}
& \left( \begin{array}{cc} 1 & 1 \\ 0 & 0 \\ \end{array} \right) &
\left( \begin{array}{cc} 2 & 1 \\ 0 & 1 \\ \end{array} \right) &
\left( \begin{array}{cc} 1 & 0 \\ 0 & 1 \\ \end{array} \right) \\ \\
\left( \begin{array}{cc} 1 & 1 \\ 0 & 0 \\ \end{array} \right) &
\left( \begin{array}{cc} 4 & 2 \\ 0 & 2 \\ \end{array} \right) &
\left( \begin{array}{cc} 5 & 1 \\ 0 & 3 \\ \end{array} \right) &
\left( \begin{array}{cc} 2 & 0 \\ 0 & 1 \\ \end{array} \right) \\ \\
\left( \begin{array}{cc} 2 & 1 \\ 0 & 1 \\ \end{array} \right) &
\left( \begin{array}{cc} 5 & 1 \\ 0 & 3 \\ \end{array} \right) &
\left( \begin{array}{cc} 4 & 0 \\ 0 & 2 \\ \end{array} \right) &
\left( \begin{array}{cc} 1 & 0 \\ 0 & 0\\ \end{array} \right) \\ \\
\left( \begin{array}{cc} \bf{1} & \bf{0} \\ \bf{0} & \bf{1} \\
\end{array} \right) &
\left( \begin{array}{cc} 2 & 0 \\ 0 & 1 \\ \end{array} \right) &
\left( \begin{array}{cc} 1 & 0 \\ 0 & 0 \\ \end{array} \right) &
\end{array} \ .
$$
The optimal difference mask is given by its symbol
$B^*(z)=\frac{1}{8}\left( b_{ij}(z) \right)_{1 \le i,j \le 4}$
with integral entries
\begin{eqnarray*}
 b_{11}(z)&=&(1+z_2)^2(1+z_1)(z_2z_1+1), \quad b_{12}(z)=(1-z_1)a_{11}(z), \\
 b_{13}(z)&=&b_{14}(z)=0, \quad
 b_{21}(z)=0, \quad b_{22}(z)=a_{22}(z), \quad b_{23}(z)=b_{24}(z)=0, \\
 b_{31}(z)&=&0, \quad b_{12}(z)=(1-z_2)a_{11}(z), \quad
 b_{33}(z)=(1+z_2)(1+z_1)^2(z_2z_1+1),\\
 b_{34}(z)&=&0, \quad
 b_{41}(z)=0, \quad b_{42}(z)=a_{22}(z), \quad b_{43}(z)=b_{44}(z)=0, \\
\end{eqnarray*}

The entries $b_{11}, b_{13}, b_{31}$ and $b_{33}$ are computed
using the optimization problem in \eqref{new} for the scalar mask
given by $a_{11}(z)$, which is defined by the $A_{11}(\alpha)$
elements of $\Ab$. If we start the flow algorithm from section
\ref{section:algorithms_and_examples}
with $d$ derived from this optimal $b_{11}$, $b_{13}$, $b_{31}$ and $b_{33}$,
we get $d^*=d$ as an output. The rest of the entries in $B^*(z)$ are defined
so that the associated operator $S_{\Bb^*}$ satisfies
$$
 \left(\begin{array}{cc} \nabla_1 & 0 \\ 0 & 1\\ \nabla_2 & 0 \\ 0 &1
 \end{array}\right)S_{\Ab}=S_{\Bb^*} \left(\begin{array}{cc} \nabla_1 & 0 \\ 0 & 1
 \\ \nabla_2 & 0 \\ 0 &1
 \end{array}\right)
$$
with $\nabla_\ell$, $\ell=1,2$, defined in \eqref{def:nabla_k},
see \cite{CDM91,C11} for more details on the structure of the
difference operator $\nabla$ and difference masks in the vector
case. For the optimal mask $\Bb^*$ we have
$$
 \|S_{\Bb^*}|_\nabla\|_\infty=\|S_{\Bb^*}\|_\infty=\frac{3}{4}.
$$
\end{example}

In the next example we determine an optimal second difference mask
for the so-called butterfly scheme studied in e.g. \cite{DLM}.

\begin{example} The dilation matrix is $M=2I$ and the mask is
given by its symbol
$$
 a(z)=\frac{1}{2}(z_1+1)(z_2+1)(z_1z_2+1)(z_1^2z_2^2-\frac{1}{16}c(z)), \quad z \in (\CC \setminus 0)^2,
$$
where
$$
 \begin{array}{c} c(z)=2z_2+2z_1-4z_1z_2-4z_1z_2^2-4z_1^2z_2+2z_1z_2^3+2z_1^3z_2+
    12z_1^2z_2^2\\-4z_1^3z_2^2-4z_1^2z_2^3-4z_1^3 z_2^3+2z_1^4
    z_2^3+2z_1^3z_2^4\ .\end{array}
$$
To show the $C^1-$regularity of the butterfly scheme by solving
the optimization problem in \eqref{new}, we have to determine
an optimal mask for the third iteration of the second difference
operator. An optimal mask for the second difference scheme is
determined easily, if for its derivation, instead of $\nabla^2$ in
\eqref{def:nabla_k}, we  make use of all three factors
$(z_1+1)(z_2+1)(z_1z_2+1)$ as it is done in \cite{DLM}. The
diagonal structure of the symbol $B^*(z)=\frac{1}{16}\left(
b_{ij}(z)\right)_{1\le i,j \le 3}$ with non-zero elements
\begin{eqnarray*}
 b_{11}(z)&=&(1+z_1)^{-1}(1+z_1z_2)^{-1}A(z), \quad
 b_{22}(z)=(1+z_1)^{-1}(1+z_2)^{-1}A(z),\\
 b_{33}(z)&=&(1+z_2)^{-1}(1+z_1z_2)^{-1}A(z),
\end{eqnarray*}
implies that the corresponding mask is optimal. This special
structure of the symbol allows us to use the univariate strategy
in section \ref{subsec:univariate_case} to show the optimality of
the mask. The iterates of $S_{\Bb^*}$ are also optimal and
$\|S^2_{\Bb^*}\|_\infty<1/2$ implies that the scheme is $C^1$.
\end{example}

The last example shows that although the optimal mask determined
by solving the optimization problem in Theorem
\ref{th:existence_optimal_mask_general_case} can be non-integral,
the optimal value $z^*$ still is. We were not able to find
subdivision schemes with integral masks (after an appropriate
normalization), which did not possess either integral optimal
masks for higher order difference schemes or for which $z^*$ were
not integral.

\begin{example} Let $M=2I$ and
$$
a(z)=4\left(\frac{1+z_1}{2}\right)\left(\frac{1+z_2}{2}\right)^3
           \left(\frac{1+z_1z_2}{2}\right)^3.
$$
The associated bivariate scheme  generates a three-directional box
spline. Matlab yields $\|S_{\Bb_2^*}\|_\infty=z^*=\frac{3}{8}<
\frac{1}{2}$ implying the $C^1-$regularity of the scheme. The optimal second difference mask
$2^5\Bb_2^* \in \ell_0^{3 \times 3}(\ZZ^2)$ satisfying \eqref{id:Sa=Sbk} is not integral and we
think it will serve no purpose to present it here. However, it
allows us to derive  another optimal second difference mask  given
by $B_2^*(z)= 2^{-5} \left(b_{ij}(z)\right)_{i,j=1,\dots,3}$ with
integral
\begin{eqnarray*}
 b_{11}(z)&=& z_2(z_2+1)(z_1^2z_2^4+2z_1^2z_2^3+5z_1z_2^2+z_1^2z_2^2+
 z_2+3z_1z_2+3), \\
 b_{12}(z)&=&(1-z_1)(z_1z_2^2+1)^2, \quad  b_{13}(z)=0,
 \quad  b_{21}(z)=0 \\
 b_{22}(z)&=&(z_2+1)^2(z_1z_2+1)^3, \quad b_{31}(z)=z_2^4-z_2^6, \\
 b_{32}(z)&=&z_2^2(-1+z_1^2z_2^4-z_1+z_2^4z_1-4z_1z_2+4z_1z_2^3),\\
 b_{33}(z)&=&(1+z1)(z_1^3z_2^3+(2z_2^2-z_2^3)z_1^2+(3z_2^2+
 4z_2^3+z_2^4+3z_2)z_1+1+z_2^2+z_2).
\end{eqnarray*}

\end{example}

\section{Summary}

In this paper we establish a link between convergence analysis of
multivariate subdivision schemes and network flows as well as between the regularity analysis of subdivision and
linear optimization. Advances in network flow theory and linear optimization provide efficient algorithms for
determining what we call optimal difference masks. The regularity of the underlying
subdivision scheme can be easily read off the corresponding optimal values, which determine the norm of the difference operators defined by
such optimal difference masks. We would like to emphasize that we only prove the existence of the optimal masks, which are
by no means unique. The existence of the optimal masks shows that there are no conceptual differences in the analysis of multivariate and univariate subdivision schemes. Moreover, if the subdivision mask has only rational entries, then so does the first difference optimal mask.
There are no theoretical results that guarantee the same in the case of higher order difference masks, but we were not
able to find an example of a subdivision scheme with rational entries whose higher order difference schemes would be irrational.


\end{document}